\documentclass[a4paper,twoside]{amsart}
\usepackage{amsmath,amssymb,amsthm,dsfont,graphicx}
\usepackage[colorlinks=true, linkcolor=black, citecolor=blue, urlcolor=red]{hyperref}
\usepackage[english]{babel}
\usepackage[lite]{amsrefs}

\newcommand   {\greeke}      {\varepsilon                                      }

\newcommand   {\greekee}     {\eta                                             }

\newcommand   {\greekl}      {\lambda                                          }
\newcommand   {\greekL}      {\Lambda                                          }

\newcommand   {\greekp}      {\pi                                              }

\newcommand   {\greekph}     {\varphi                                          }

\newcommand   {\greekO}      {\Omega                                           }
\newcommand   {\scripta}     {{\scalebox{1.10}{$\mathfrak{a}$}}                }

\newcommand   {\scriptb}     {{\scalebox{1.10}{$\mathfrak{b}$}}                }

\newcommand   {\scriptB}     {\mbox{$\mathfrak{B}$}                            }

\newcommand   {\scriptI}     {\mbox{$\mathfrak{I}$}                            }

\newcommand   {\scriptm}     {{\scalebox{1.10}{$\mathfrak{m}$}}                }

\newcommand   {\scriptN}     {\mbox{$\mathfrak{N}$}                            }

\newcommand   {\scriptp}     {{\scalebox{1.10}{$\mathfrak{p}$}}                }

\newcommand   {\scriptq}     {{\scalebox{1.10}{$\mathfrak{q}$}}                }

\newcommand   {\scriptu}     {{\scalebox{1.10}{$\mathfrak{u}$}}                }

\newcommand   {\scriptU}     {\mbox{$\mathfrak{U}$}                            }
\newcommand   {\smallscriptU}{{\scalebox{0.70}{$\mathfrak{U}$}}                }
\newcommand   {\scriptv}     {{\scalebox{1.10}{$\mathfrak{v}$}}                }

\newcommand   {\scriptW}     {\mbox{$\mathfrak{W}$}                            }

\newcommand   {\setC}        {\mathds{C}                                       }

\newcommand   {\setI}        {\mathds{I}                                       }

\newcommand   {\setN}        {\mathds{N}                                       }

\newcommand   {\setQ}        {\mathds{Q}                                       }

\newcommand   {\setV}        {\mathds{V}                                       }

\newcommand   {\Where}       {\mbox{ where }                                   }

\newcommand   {\SuchThat}    {\mbox{ such that }                               }

\newcommand   {\A}           {\forall \,                                       }
\newcommand   {\E}           {\exists \,                                       }

\newcommand   {\noneq}       {\not =                                           }
\newcommand   {\nonin}       {\not \in                                         }
\renewcommand {\implies}     {\, \Longrightarrow \,                            }

\newcommand   {\set}[1]      {\left\{ #1 \right\}                              }

\newcommand   {\setin}       {\subseteq \,                                     }

\newcommand   {\setinn}      {\subset                                          }

\newcommand   {\C}           {\scalebox{0.9}{$\complement$}                    }

\newcommand   {\onto}        {\, \twoheadrightarrow \,                         }

\newcommand   {\isoto}       {\, \stackrel{\raisebox{-0.5ex
                                 }[0pt]{$\sim$}}{\longrightarrow} \,           }

\newcommand   {\abs}[1]      {| #1 |                                           }
\newcommand   {\gen}[1]      {\langle #1 \rangle                               }

\newcommand   {\kn}          {\mbox{\rm{ker}}                                  }

\newcommand   {\ideal}       {\unlhd                                           }
\newcommand   {\Modulo}[2]   {\raisebox{3pt}{$#1$} \hspace{-1pt}
                              \raisebox{-1pt}{\scalebox{1.5}{/}}\hspace{1dd}
															\hspace{-2pt} \raisebox{-3pt}{$#2$}              }
\newcommand   {\spec}        {\mbox{\rm{Spec}}                                 }

\newcommand   {\dirlim}      {\varinjlim                                       }
\newcommand   {\invlim}      {\varprojlim                                      }
\newcommand   {\mapatop}[1]  {\stackrel{\hbox{\scriptsize ${#1}$}}{
                              \longrightarrow}                                 }
\newcommand   {\mapbelow}[1] {\stackrel{\longrightarrow}{\hbox{\lower6pt
                              \hbox{\scriptsize ${#1}$}}}                      }

\newcommand   {\mapdown}[1]  {\hspace{0.25cm} \hbox{$\downarrow \raise1pt
                              \hbox{\scriptsize ${#1}$}$}                      }

\newcommand   {\mapsdown}    {\rotatebox[origin=c]{270}{\mbox{$\mapsto$}}      }

\newcounter{claimcounter}

\newtheorem{theorem}{Theorem}[section]
\newtheorem{proposition}[theorem]{Proposition}
\newtheorem{corollary}[theorem]{Corollary}
\newtheorem{lemma}[theorem]{Lemma}

\theoremstyle{definition}
\newtheorem{definition}[theorem]{Definition}
\newtheorem{example}[theorem]{Example}
\newtheorem{remark}[theorem]{Remark}
\allowdisplaybreaks

\title[Quasi-Compactness in Infinite Dimension]{Quasi-Compactness in Infinite Dimension}

\author[A. Bernhard Zeidler]{A.~Bernhard Zeidler}

\address{Mathematisches Institut, Universit\"at T\"ubingen,
Auf der Morgenstelle 10, 72076 T\"ubingen, Germany}
\email{zeidler@math.uni-tuebingen.de}


\begin{document}

\begin{abstract}
 We give extensive characterizations for an open subset of an affine space of arbitrary
 dimension and, respectively, of an inverse limit of prime spectra to be quasi-compact.
 Among other things weak stability, retro-compactness, and cylinder sets provide equivalent
 criteria in both settings. We also exhibit an example of a non–quasi-compact affine
 space.
\end{abstract}

\maketitle

\section{The main Results}
\label{sec:mainresults}

 In this article, we investigate quasi-compactness of open subsets $U$ in two different
 settings of inverse limits of topological spaces: In affine spaces $F^L$ of arbitrary
 dimension $L$ over an algebraically closed field $F$ on one hand, and in inverse
 limits of prime spectra of noetherian rings on the other. Our main results (Theorems
 \ref{thm:affineresult} and \ref{thm:specquasicompact}) provide extensive lists of
 equivalent conditions characterizing the quasi-compactness of $U$ in these settings.
 Our aim is to provide a synopsis that highlights similarities and differences of these
 settings through complete, elementary proofs.

 Among other things, we relate standard concepts and ideas from the theory of arc
 spaces and motivic integration such as \emph{weak stability} see, e.g.~Denef-Loeser
 \cite[Def.~2.4]{DeLo}, and \emph{cylinder sets} as used by Veys \cite[Def.~3.1]{Veys}
 and others. Furthermore, our characterizing condition (f) in Theorem \ref{thm:affineresult}
 is a necessary requirement in the definition of \emph{stable sets} as given in
 Campesato-Fukui-Kurdyka-Parusiński \cite[Def.~3.33]{CFKP} or Veys \cite[Def.~3.2]{Veys}.
 Our aim is to present a unified, self-contained treatment and a consolidated survey.

 We come to our first, original result. There and throughout the whole article, the
 affine space $F^L$ is endowed with the Zariski topology, that means that its closed
 sets are common zero-loci of sets of polynomials of $F[t_i \mid i \in L]$.

 \begin{theorem} \label{thm:affineresult}
 Let $L$ be any set and $F$ an algebraically closed field of cardinality $\abs{L} <
 \abs{F}$. For $I \setin L$ let $\greekp_I \colon F^L \onto  F^I$ be the canonical
 projection. Let $C \setin F^L$ be a closed set, $\scriptu := \setI(C)$ its vanishing
 ideal in $F[t_i \mid i \in L]$ and $U := F^L \setminus C$ its (open) complement.
 Then the following statements are equivalent:
 \begin{enumerate}
  \item[(a)] $\scriptu$ is finitely generated.
  \item[(b)] $U$ is quasi-compact.
	\item[(c)]
	 $U$ is a \emph{cylinder set}, that is there is a finite subset $I \setin L$ and
	 a quasi-con\discretionary{-}{}{}structible subset $W \setin F^I$ such that $U$
	 is the preimage $U = \greekp_I^{-1}(W)$.
	\item[(d)]
	 $U$ is \emph{retro-compact}, that is for any open, quasi-compact subset $V \setin
	 F^L$ the intersection $U \cap V$ is quasi-compact.
	\item[(e)]
	 Given a set $\greekL$ and for any $\greekl \in \greekL$ a finite set $I(\greekl)
	 \setin L$ and an open set $A_{\greekl} \setin F^{I(\greekl)}$ such that $U$ is
	 covered by the open sets
	 \[ U \ = \ \bigcup_{\greekl \in \greekL} W_{\greekl} \ \ \Where \ \
	    W_{\greekl} \ := \ \greekp_{I(\greekl)}^{-1} \big( A_{\greekl} \big) \]
	 this cover admits a sub-cover consisting only of open base sets of a bounded level.
	 That is, for some finite set $K \setin L$ we have
   \begin{eqnarray*}
	  U          &  = & \bigcup_{\greekl \in \greekL[K]} W_{\greekl} \ \ \Where          \\
	  \greekL[K] & := & \set{\greekl \in \greekL \mid \E B_{\greekl} \setin F^K
		                  \mbox{ open} \ : \ W_{\greekl} = \greekp_{K}^{-1} \big( B_{\greekl})}.
	 \end{eqnarray*}
	\item[(f)]
	 There is a finite subset $J \setin L$ such that $U = \greekp_J^{-1}
	 \big( \greekp_J(U) \big)$.
	\item[(g)]
	 There is a finite subset $J \setin L$ such that $C = \greekp_J^{-1}
	 \big( \greekp_J(C) \big)$.
	\item[(h)]
	 $U$ is \emph{weakly stable}, that is there are a finite subset $K \setin L$ and
	 an open subset $U_K \setin F^K$ such that $U = \greekp_K^{-1} \big( U_K \big)$.
	\item[(i)]
	 There are a finite subset $K \setin L$ and a closed subset $C_K \setin F^K$
	 such that $C = \greekp_K^{-1} \big( C_K \big)$.
 \end{enumerate}
\end{theorem}

 In particular (under the assumptions of \ref{thm:affineresult}) we find that $F^L$
 is quasi-compact, as it is weakly stable. The assumption that the cardinality of
 $F$ exceeds that of $L$ is needed in this theorem: In Example \ref{xmp:notquasicompact}
 we show that, \emph{when $F$ is the field of algebraic numbers, the affine space
 $F^\setN$ is not quasi-compact}.

 Note that $F^L$ can be regarded as the inverse limit of the $F^I$, where $I$ ranges
 over the finite subsets of $L$, and the limit topology of the $F^I$ coincides with the
 Zariski topology of $F^L$. From this point of view, Theorem~\ref{thm:affineresult} is a
 statement on inverse limits of finite-dimensional affine spaces $F^n$, in other words,
 the maximal spectra of the polynomial rings $F[t_1,\ldots,t_n]$. If we consider the
 prime spectrum of $F[t_1,\ldots,t_n]$ instead, then we are in our second setting:
 Inverse limits of prime spectra of noetherian commutative rings; also see
 Remark~\ref{rem:closingremarks} for further discussion.

 Our second theorem concerns the corresponding properties for noetherian, affine
 schemes. This is a special case of the situation in EGA (spectra of quasi-coherent
 algebras over a fixed scheme). We give direct elementary proofs in this setting,
 avoiding an appeal to the full EGA formalism.
 
\begin{theorem} \label{thm:specquasicompact}
 Let $(I,\leq)$ be a net, $\big( R_i, \greekph_i^j \big)$ a direct system of noetherian
 commutative rings over $I$ and $R_{\infty}$ the associated direct limit. For $i \leq j
 \in I$ denote $X_i := \spec(R_i)$ and $f_j^i := \spec(\greekph_i^j)$. Denote by $X_{
 \infty}$ the inverse limit of the system $\big( X_i, f_j^i \big)$ of topological spaces.
 Then, for any open subset $U \setin X_{\infty}$, the following statements
 are equivalent:
 \begin{enumerate}
  \item[(a)] $U$ is quasi-compact.
	\item[(b)] $U$ is quasi-stable in the sense of Def.~\ref{def:quasistable}.
	\item[(c)] $U$ is weakly stable in the sense of Def.~\ref{def:weaklystable}.
	\item[(d)] $U$ is retro-compact.
	\item[(e)] $U$ is a cylinder set in the sense of Def.~\ref{def:cylinderset}.
	\item[(f)]
   There is a quasi-finite ideal $\scriptu \ideal R_{\infty}$, see
	 Def.~\ref{def:quasifinite}, such that $U = X_{\infty} \setminus \setV(\scriptu)$.
 \end{enumerate}
\end{theorem}

 The two main results are proven in Sections~\ref{sec:proofs1} and~\ref{sec:proofs2},
 respectively. Moreover, in the latter, we present Corollary~\ref{cor:quasicompact},
 extending Corollaire 8.2.11 of~\cite{EGA4.3}. Continuing in this vein, Corollary
 \ref{cor:constructiblecompact} provides: If a cylinder set is covered by an arbitrary
 family of cylinder sets, then it is already the union of finitely many of them.
 Note that this generalizes~\cite[Lemma~2.4]{DeLo}. In the final Section
 \ref{sec:counterexamples}, we present noteworthy examples that contrast our two
 settings of inverse limits with each other.

\section{Proof of Theorem~\ref{thm:affineresult}}
\label{sec:proofs1}

 For any commutative ring $R$ and any set $I$ we abbreviate the polynomial ring $R[I]
 := F[t_i \mid i \in I]$ in the variables $t_i$ over $R$. We need some preparations:

 \begin{proposition} \label{pro:polyradical}
  Let $R$ be a commutative ring, $\scripta \ideal R$ and $S := R[I]$ for some set
	$I$. Then we get the identity $\sqrt{\scripta} \, S = \sqrt{\scripta S}$ of ideals
	in $S$.
 \end{proposition}

 \begin{proof}[Proof of~\ref{pro:polyradical}]
  The inclusion $\sqrt{\scripta} \, S \setin \sqrt{\scripta S}$ is true for any ring
	extension $R \setin S$. For the converse inclusion let $\scriptb := \sqrt{\scripta}$
	and consider the isomorphism of rings
	\[ \Modulo{S}{\scriptb S} \ = \ \Modulo{R[I]}{\scriptb R[I]} \ \cong \
	   \left( \Modulo{R}{\scriptb} \right)[I] \]
	As $\scriptb$ is a radical ideal $R/\scriptb$ is reduced. Hence its polynomial
	ring is reduced, too. Using the isomorphism this implies $\scriptb S$ to be a
	radical ideal again. As $\scripta S \setin \scriptb S$ we conclude $\sqrt{\scripta
	S} \setin \sqrt{\scriptb S} = \scriptb S = \sqrt{\scripta} S$.
 \end{proof}

 For the proof Theorem \ref{thm:affineresult} we need to study ideals in the polynomial
 ring $F[L]$ which is the direct limit $F[L] = \dirlim F[I]$ of polynomial rings in
 finitely many variables $I \in \scriptI := \set{I \setin L \mid I \mbox{ finite}}$.
 The second ingredient is an algebraic version of the compactness property:

\goodbreak
\begin{definition} \label{def:quasifinite}
 Let $(I,\leq)$ be a net and $\big( R_i, \greekph_i^j \big)$ a direct system of
 commutative rings over $I$. We abbreviate its direct limit by $R_{\infty}$. An
 ideal $\scriptu \ideal R_{\infty}$ is said to be {\bf quasi-finite of level}
 $i \in I$, if it is of the form $\scriptu = \scripta R_{\infty}$ for some ideal
 $\scripta \ideal R_i$. And $\scriptu$ is said to be {\bf quasi-finite}, if it
 is quasi-finite of some level $i \in I$.
\end{definition}

\begin{definition} \label{def:uinfty}
 We continue with the setting of \ref{def:quasifinite} and regard an arbitrary set
 $\greekL \noneq \emptyset$ and for any $\greekl \in \greekL$ some quasi-finite
 ideal $\scriptu_{\greekl} \ideal R_{\infty}$. For any $i \in I$ let $\greekL[i]$
 be the set of all indices $\greekl$ of ideals $\scriptu_{\greekl}$ that are quasi-finite
 of level $i$, that is
 \[ \greekL[i] \ := \ \set{\greekl \in \greekL \mid \E \scriptb_i \ideal R_i
    \, : \, \scriptu_{\greekl} = \scriptb_i R_{\infty}}. \]
 Then we will denote the following ideals (for any $i \in I$) of the direct limit
 $R_{\infty}$
\[ \scriptu[i] \ := \ \sum_{\greekl \in \greekL[i]} \, \scriptu_{\greekl}, \hspace{1.5cm}
   \scriptu[\infty] \ := \ \sum_{\greekl \in \greekL}    \, \scriptu_{\greekl}. \]
\end{definition}

\begin{proposition} \label{pro:limcompact}
 Let $(I,\leq)$ be a net, $\big( R_i, \greekph_i^j \big)$ be a direct system of
 commutative rings over $I$ and $R_{\infty}$ its direct limit. Let $\scriptu_{\greekl}
 \ideal R_{\infty}$ be a family (where $\greekl \in \greekL$) of quasi-finite ideals.
 Then we get the following properties for the ideals defined above:
 \begin{enumerate}
  \item[(i)]
	 For any $i \leq j \in I$ we have $\greekL[i] \setin \greekL[j]$ and $\scriptu[i]
	 \setin \scriptu[j]$.
  \item[(ii)]
	 $\scriptu[i]$ is quasi-finite of level $i$ and $\scriptu[\infty]$ is the union
	 of the $\scriptu[i]$.
	\item[(iii)]
	 For any ideal $\scriptu \ideal R_{\infty}$ we have the equivalence of (a) $\scriptu$
	 is finitely generated and (b) there are some $k \in I$ and some finitely generated
	 ideal $\scriptb \ideal R_k$ such that $\scriptu = \scriptb R_{\infty}$. In
	 particular finitely generated implies quasi-finite.
	\item[(iv)]
	 If $\scriptu[\infty]$ is finitely generated, then $\scriptu[\infty] = \scriptu[m]$
	 for some $m \in I$.
  \item[(v)]
	 If $\scriptv \ideal R_{\infty}$ is some finitely generated ideal, then we obtain
	 the implication
   \[ \sqrt{\scriptu[\infty]} \, = \, \sqrt{\scriptv} \ \implies \
      \E m \in I \ \colon \ \sqrt{\scriptu[\infty]} \, = \, \sqrt{\scriptu[m]}. \]
 \end{enumerate}
\end{proposition}

\begin{proof}[Proof of~\ref{pro:limcompact}]
 For (i) consider any $\greekl \in \greekL[i]$, this means $\scriptu_{\greekl} =
 \scripta_{\greekl} R_{\infty}$ for some $\scripta_{\greekl} \ideal R_i$. As $i \leq j$
 we may take $\scriptb_{\greekl} := \scripta_{\greekl} R_j$, then $\scriptb_{\greekl}
 R_{\infty} = \scripta_{\greekl} R_j R_{\infty} = \scripta_{\greekl} R_{\infty} =
 \scriptu_{\greekl}$. In particular $\greekl \in \greekL[j]$ and thereby $\greekL[i]
 \setin \greekL[j]$. Property (ii) follows from the fact that the extension of ideals
 commutes with taking sums of ideals: We find $\scriptu[i]$ to be quasi-finite of
 level $i$ from
 \[ \scriptu[i] \ = \ \sum_{\greekl \in \greekL[i]} \scriptu_{\greekl} \ = \
    \sum_{\greekl \in \greekL[i]} \scripta_{\greekl} R_{\infty}  \ = \ \left(
		\sum_{\greekl \in \greekL[i]} \scripta_{\greekl} \right) R_{\infty}. \]
 Next we prove, that $\scriptu[\infty]$ is the union of the $\scriptu[i]$: For any
 $i \in I$ we have $\scriptu[i] \setin \scriptu[\infty]$ and thereby the union of the
 $\scriptu[i]$ is contained in $\scriptu[\infty]$. For the converse inclusion regard
 $f \in \scriptu[\infty]$, say $f = f_1 + \dots + f_s$. For any $r \in 1,\ldots,s$
 there is some $\greekl(r) \in \greekL$, such that $f_r \in \scriptu_{\greekl(r)}$.
 By assumption $\scriptu_{\greekl(r)}$ is quasi-finite of some level $i(r) \in I$.
 Choose $k \in I$ with $k \geq i(r)$ for all $r \in 1,\ldots,s$, then $f_r \in
 \scriptu_{\greekl(r)} \setin \scriptu[i(r)] \setin \scriptu[k]$. As $\scriptu[k]$
 is an ideal, we get $f = f_1 + \dots + f_r \in \scriptu[k]$.

 Let us prove the equivalence in (iii): If $\scriptu = \scriptb R_{\infty}$ for some
 finitely generated ideal $\scriptb = \gen{b_1,\ldots,b_s} \ideal R_k$, then
 $\scriptu$ is generated by $[k,b_r] \in R_{\infty}$ where $r \in 1,\ldots,s$ and
 hence is finitely generated, again. Conversely, if $\scriptu = \gen{u_1,\ldots,u_s}$
 for some $u_r \in R_{\infty}$, then any $u_r$ is of the form $u_r =[i(r),a_r]$,
 where $i(r) \in I$ and $a_r \in R_{i(r)}$. Choose $k \in I$ with $k \geq i(r)$ for
 all $r \in 1,\ldots,s$ and let $b_r := \greekph_{i(r)}^k(a_r) \in R_k$. Then $u_r
 = [i(r),a_r] = [k,b_k]$ and hence $\scriptu \setin \scriptb R_{\infty}$ for $\scriptb
 := \gen{b_1,\ldots,b_s}$. On the other hand $[k,b_r] = u_r \in \scriptu$ is clear
 and hence $\scriptb R_{\infty} \setin \scriptu$, as well.

 In the proof of (iv) we can use (iii) to find some $k \in I$ and a finitely generated
 ideal $\scriptb \ideal R_k$ such that $\scriptu[\infty] = \scriptb R_{\infty}$, say
 $\scriptb = \gen{b_1,\ldots,b_s}$ for some $b_r \in R_k$. Then $\scriptu[\infty] =
 \scriptb R_{\infty}$ is generated by the elements $\widehat{b}_r := [k,b_r]$. By
 (ii) $\scriptu[\infty]$ is the union of the $\scriptu[i]$, that is there are some
 $i(r) \in I$ such that $\widehat{b}_r \in \scriptu[i(r)]$. Choose $m \in I$ with
 $m \geq i(1),\ldots,i(s)$. Then $\widehat{b}_r \in \scriptu[i(r)] \setin \scriptu[m]$
 such that we find the chain
	 \[ \scriptu[\infty] \ = \ \gen{\widehat{b}_1,\ldots,\widehat{b}_s} \ \setin \
	    \scriptu[m] \ \setin \ \scriptu[\infty]. \]
 The proof of (v) proceeds similarly: As $\scriptv \ideal R_{\infty}$ is finitely
 generated we find another ideal $\scriptb = \gen{b_1,\ldots,b_s} \ideal R_k$ such
 that $\scriptv = \scriptb R_{\infty}$. Then $\scriptv = \scriptb R_{\infty}$ is
 generated by $\scriptb R_{\infty} = \gen{\widehat{b}_1,\ldots,\widehat{b}_s}$ again,
 where $\widehat{b}_r := [k,b_r]$. By assumption we have $\sqrt{\scriptu[\infty]} =
 \sqrt{\scriptv}$ such that $\widehat{b}_r \in \scriptb R_{\infty}$ implies $\widehat{
 b}_r \in \sqrt{\scriptb R_{\infty}} = \sqrt{\scriptv} = \sqrt{\scriptu[\infty]}$.
 From this we get
 \[ \A r \in 1,\ldots,s \ \E p(r) \in \setN \ \ \SuchThat \ \
	  \widehat{b}_r^{p(r)} \, \in \ \scriptu[\infty]. \]
 As $\scriptu[\infty]$ is the union of the $\scriptu[i]$ there are some $i(r) \in I$
 such that $\widehat{b}_r^{p(r)} \in \scriptu[i(r)]$. Choose $m \in I$ with $m \geq
 i(1),\ldots,i(s)$. Then $\widehat{b}_r^{p(r)} \in \scriptu[i(r)] \setin \scriptu[m]$
 by (ii), which is $\widehat{b}_r \in \sqrt{\scriptu[m]}$. As these elements generate
 $\scriptv$, we get $\scriptv \setin \sqrt{\scriptu[m]}$ and hence
 \[ \sqrt{\scriptu[\infty]} \ = \ \sqrt{\scriptv} \ \setin \ \sqrt{\scriptu[m]}
	  \ \setin \ \sqrt{\scriptu[\infty]}. \]
\end{proof}

 These are the algebraic ingredients for the proof of \ref{thm:affineresult}. We
 will also use property (i) in~\cite[Thm.~1.1]{Zeid1}, stating $\setI \setV(\scriptu)
 = \sqrt{\scriptu}$ for any ideal $\scriptu \ideal F[L]$. It requires $\abs{L} <
 \abs{F}$. Claim 6 in that paper also contains the identities $\setV \big( \scripta
 F[L] \big) = \greekp_I^{-1} \big( \setV(\scripta) \big)$ and $\setI \big( \greekp_I^{
-1}(X) \big) = \setI(X) F[L]$ for any $I \setin L$ and any ideal $\scripta \ideal
 F[I]$ resp.~subset $X \setin F^I$.

\begin{proof}[Proof of Theorem~\ref{thm:affineresult}]
 {\bf Part 1}: (f) to (i) are equivalent. The implication (h)$\implies$(f) is generally
 true: For any function $f \colon X \to Y$ and any subset $B \setin Y$ we have $f^{-1}
 \big( f \big( f^{-1}(B) \big) \big) = f^{-1}(B)$. The implication (f)$\implies$(g) also is
 generally true, just go to complements. In the next step (g)$\implies$(i) we use the
 above identities: We have $C = \greekp_J^{-1} \big( \greekp_J(C) \big)$, let $\scriptb
 := \setI \big( \greekp_J(C) \big) \ideal F[J]$, then we get
 \[ \scriptu \ = \ \setI(C) \ = \ \setI \Big( \greekp_J^{-1} \big( \greekp_J(C)
	  \big) \Big) \ = \ \setI \big( \greekp_J(C) \big) F[L] \ = \ \scriptb \, F[L]. \]
 Take $K := J$ and $C_K := \setV(\scriptb)$. Then $C_K \setin F^K$ is closed and
 $C = \setV \setI(C) = \setV(\scriptb F[L]) = \greekp_K^{-1}(C_K)$. It remains to
 prove (i)$\implies$(h): As $C_K \setin F^K$ is closed, $U_K := F^K \setminus C_K$
 is open and it is straightforward to see, that $U$ is the preimage of this set:
 \[ U \ = \ F^L \setminus C \ = \ F^L \setminus \greekp_K^{-1} \big( C_K \big) \ = \
    \greekp_K^{-1} \Big( F^K \setminus C_K \Big) \ = \ \greekp_K^{-1} \Big( U_K \Big). \]
 {\bf Part 2}: (b), (e) and (h) are equivalent. The implications (b)$\implies$(e)
 and (e)$\implies$(h) are true in any inverse limit of topological spaces: Just
 choose a cover of $U$ by open base sets - these are of the form $\greekp_I^{-1}(
 A)$ for some $I \in \scriptI$ and $A \setin F^I$ open. As $U$ is quasi-compact
 finitely many of these suffice. And a finite union of open base sets is an open
 base set again. The major step here is (h)$\implies$(b): We need to show that $U
 \setin F^L = \invlim F^I$ is quasi-compact, so let $W_{\greekl}$ (where $\greekl
 \in \greekL$) be an open cover of~$U$. We may assume the $W_{\greekl}$ to be open
 open base sets, i.e.~there are open sets $A_{\greekl} \setin F^{I(\greekl)}$
 such that $W_{\greekl} = \greekp_{I(\greekl)}^{-1}(A_{\greekl})$. As $A_{\greekl}
 \setin F^{I(\greekl)}$ is an open set in the Zariski topology,
 it is the complement of some algebraic set $\setV(\scripta_{\greekl})$ where
 $\scripta_{\greekl} \ideal F[I(\greekl)]$ is an ideal in some polynomial ring.
 Denote $\scriptu_{\greekl} := \scripta_{\greekl} F[L]$, then $U$ is the complement
 of an algebraic set, as well:
 \begin{eqnarray*}
  U & = & \bigcup_{\greekl \in \greekL} W_{\greekl}
    \ =\  \bigcup_{\greekl \in \greekL} \greekp_{I(\greekl)}^{-1}
		      \Big( \C \, \setV(\scripta_{\greekl}) \Big)
    \ = \ \C \, \bigcap_{\greekl \in \greekL} \greekp_{I(\greekl)}^{-1}
		      \Big( \setV(\scripta_{\greekl}) \Big)                                        \\
    & = & \C \, \bigcap_{\greekl \in \greekL} \setV \big( \scripta_{\greekl} F[L] \big)
    \ = \ \C \, \setV \left( \sum_{\greekl \in \greekL} \scriptu_{\greekl} \right)
    \ = \ \C \, \setV \big( \scriptu[\infty] \big).
 \end{eqnarray*}
 But $U$ was assumed to be weakly stable, which means $U = \greekp_J^{-1}(B)$ for
 some open subset $B \setin F^J$. Let analogously $B = \C \, \setV(\scriptb)$, where
 $\scriptb$ is an ideal in the polynomial ring $\scriptb \ideal F[J]$. Then similarly
 \[ U \ = \ \greekp_j^{-1} \Big( \C \, \setV(\scriptb) \Big) \ = \ \C \,
    \greekp_j^{-1} \Big( \setV(\scriptb) \Big) \ = \ \C \, \setV \Big( \scriptb F[L] \Big). \]
 Comparing these representations for $U$ we find $\setV(\scriptu[\infty]) = \setV(
 \scriptb F[L])$. Both are algebraic sets in $F^L$ and $F$ is an algebraically closed
 field with $\abs{L} < \abs{F}$. Using~\cite[Thm.~1.1]{Zeid1}, we may use the strong
 Nullstellensatz to find
 \[ \sqrt{\scriptu[\infty]} \ = \ \setI \, \setV (\scriptu[\infty]) \ = \
    \setI \, \setV(\scriptb F[L]) \ = \ \sqrt{\scriptb F[L]}. \]
 Note that any $\scriptu_{\greekl}$ is quasi-finite, by construction. As $J$ is
 finite $F[J]$ is noetherian. Hence $\scriptb \ideal F[J]$ and thereby $\scriptb
 F[L]$ are finitely generated. Now we may apply \ref{pro:limcompact}.(v) to get
 some $K \in \scriptI$ such that $\sqrt{\scriptu[\infty]} = \sqrt{\scriptu[K]}$.
 Resubstituting this, we get:
 \begin{eqnarray*}
  U & = & \C \, \setV(\scriptu[\infty])
	  \ = \ \C \, \setV \left( \sqrt{\scriptu[\infty]} \right)
    \ = \ \C \, \setV \left( \sqrt{\scriptu[K]} \right)
		\ = \ \C \, \setV \Big( \scriptu[K] \Big)                                          \\
    & = & \C \, \setV \left( \sum_{\greekl \in \greekL[K]}
		      \scriptu_{\greekl} \right)
    \ = \ \C \bigcap_{\greekl \in \greekL[K]} \setV \Big( \scripta_{\greekl}
		      F[L] \Big)
		\ = \ \C \bigcap_{\greekl \in \greekL[K]} \greekp_{I(\greekl)}^{-1} \Big(
		      \setV \big( \scripta_{\greekl} \big) \Big) 			                             \\
		& = & \bigcup_{\greekl \in \greekL[K]} \greekp_{I(\greekl)}^{-1} \Big( \C \,
		      \setV(\scripta_{\greekl}) \Big)
		\ = \ \bigcup_{\greekl \in \greekL[K]} \greekp_{I(\greekl)}^{-1} \Big(
		      A_{\greekl} \Big)
		\ = \ \bigcup_{\greekl \in \greekL[K]} W_{\greekl}.
 \end{eqnarray*}
 Take a look at $\scripta_{\greekl} \ideal F[I(\greekl)]$ again: As $\greekl \in
 \greekL[K]$ there is some ideal $\scriptb_{\greekl} \ideal F[K]$ such that $\scriptb_{
 \greekl} F[L] = \scriptu_{\greekl} = \scripta_{\greekl} F[L]$. In the language of
 algebraic sets this translates into
 \begin{eqnarray*}
  W_{\greekl}
	& = & \greekp_{I(\greekl)}^{-1} \Big( A_{\greekl} \Big)
	\ = \ \greekp_{I(\greekl)}^{-1} \Big( \C \, \setV(\scripta_{\greekl}) \Big)
	\ = \ \C \, \greekp_{I(\greekl)}^{-1} \Big( \setV(\scripta_{\greekl}) \Big)          \\
	& = & \C \, \setV \Big( \scripta_{\greekl} F[L] \Big)
  \ = \ \C \, \setV \Big( \scriptb_{\greekl} F[L] \Big)
	\ = \ \C \, \greekp_{K}^{-1} \Big( \setV(\scriptb_{\greekl}) \Big)                   \\
	& = & \greekp_{K}^{-1} \Big( \C \, \setV(\scriptb_{\greekl}) \Big)
  \ = \ \greekp_{K}^{-1} \Big( B_{\greekl} \Big).
 \end{eqnarray*}
 Where $B_{\greekl} := \C \, \setV(\scriptb_{\greekl}) \setin F^K$ are open subsets.
 By now we already proved, that $U$ is quasi-stable, we proceed from here:
 \[ U \  = \ \bigcup_{\greekl \in \greekL[K]} W_{\greekl} \ = \ \greekp_K^{-1} \left(
    \bigcup_{\greekl \in \greekL[K]} B_{\greekl} \right). \]
 As $K$ is finite, as well, $F^K$ is a noetherian topological space and this means,
 that any open subset of $F^K$ is quasi-compact. Therefore the union of the $B_{
 \greekl}$ where $\greekl \in \greekL[K]$ can already be established by a finite
 set $\greekO \setin \greekL[K]$ and thereby
 \[ U \ = \ \greekp_K^{-1} \left( \bigcup_{\greekl \in \greekL[K]} B_{\greekl}
    \right) \ = \ \greekp_K^{-1} \left( \bigcup_{\greekl \in \greekO} B_{\greekl}
		\right) \ = \ \bigcup_{\greekl \in \greekO} W_{\greekl}. \]  
 {\bf Part 3}: (d) and (h) are equivalent. Quasi-compact and weakly stable are
 equivalent, due to part 2. In particular $F^L$ is quasi-compact and thereby
 (d)$\implies$(b) is clear by taking $V = F^L$. In (h)$\implies$(d) $U$ is weakly
 stable and we are given an arbitrary open, quasi-compact set $V \setin F^L$.
 This means $V$ is weakly-stable again. But the intersection $U \cap V$ remains
 weakly stable and hence quasi-compact.

 {\bf Part 4}: (a) and (b) are equivalent. Let us first assume (a): $U = F^L \setminus
 \setV(\scriptu)$, where $\scriptu \ideal F[L]$ is finitely generated, say $\scriptu
 = \gen{f_1,\ldots,f_n}$. As any polynomial $f_i \in F[L]$ has finitely many variables
 only, we have $f_i \in F[\greekO_i]$ for some $\greekO_i \in \scriptI$. Let $\greekO$
 be the union of $\greekO_1$ to $\greekO_n$. Then $\greekO \in \scriptI$ again and
 $f_i \in F[\greekO]$ for any $i \in 1,\ldots,n$ such that $\scripta := f_1
 F[\greekO] + \dots + f_n F[\greekO]$ is an ideal of $F[\greekO]$. As $\scriptu =
 f_1 F[L] + \dots + f_n F[L]$ it is clear, that $\scriptu = \scripta F[L]$. From
 this we find that $\setV(\scriptu)$ is the preimage
 \[ \setV \big( \scriptu \big) \ = \ \setV \big( \scripta F[L] \big) \ = \
    \greekp_{\greekO}^{-1} \big( \setV(\scripta) \big). \]
 In particular we find that $U$ is weakly stable (and hence quasi-compact, due to
 part 2), as it is of the form
 \[ U \ = \ \C \, \setV \big( \scriptu \big) \ = \ \C \, \greekp_{\greekO}^{-1} \big(
    \setV(\scripta) \big) \ = \ \greekp_{\greekO}^{-1} \Big( \C \, \setV(\scripta) \Big). \]
 Conversely we now start in (b): As $U \setin F^L$ is open, it is of the form $U =
 F^L \setminus \setV(\scriptu)$ for some ideal $\scriptu \ideal F[L]$ with $\sqrt{
 \scriptu} = \scriptu$. In particular $U = F^L \setminus \setV(\scriptu)$ is covered
 by the principal open subsets $B(f) = \set{x \in F^L \mid f(x) \noneq 0}$ of $F^L$,
 where $f$ runs in $\scriptu$. But as $U$ is quasi-compact, by assumption, there has
 to be a finite subset $\greekO \setin \scriptu$ such that $U$ is covered by the
 principal open sets $B(f)$ of $f \in \greekO$ only:
 \[ U \ = \ \bigcup_{f \in \greekO} \set{x \in F^L \mid f(x) \noneq 0}. \]
 Going to complements again, we see that $\setV(\scriptu)$ is the intersection of
 finitely many closed sets $\setV(f) \setin F^L$
 \[ \setV(\scriptu) \ = \ F^L \setminus U \ = \ \bigcap_{f \in \greekO}
    \set{x \in F^L \mid f(x) = 0}. \]
 As $\greekO$ is finite and any $f \in \greekO$ only has finitely many variables,
 there is some finite subset $K \setin L$, such that $f \in F[K]$ for any $f \in
 \greekO$. Let $\scripta := \gen{f \mid f \in \greekO} \ideal F[K]$ be the ideal
 generated in $F[K]$. Then we find
 \[ \setV(\scriptu) \ = \ \bigcap_{f \in \greekO} \setV(f) \ = \ \setV \left(
    \sum_{f \in \greekO} f F[L] \right) \ = \ \setV \big( \scripta F[L] \big). \]
 The assumptions on $F$ enable us to use~\cite[Theorem 1.1]{Zeid1}, with which
 we compute
 \[ \scriptu \ = \ \setI \setV(\scriptu) \ = \ \setI \setV \big( \scripta F[L]
    \big) \ = \ \sqrt{\scripta F[L]} \ = \ \sqrt{\scripta} \, F[L] \]
 Hereby $\sqrt{\scripta F[L]} = \sqrt{\scripta} \, F[L]$ is
 Proposition \ref{pro:polyradical}. So, as $K$ is finite, $F[K]$ is noetherian and hence $\sqrt{
 \scripta}$ is finitely generated. Thereby the extension $\scriptu = \sqrt{\scripta}
 F[L]$ is finitely generated, as well. {\bf Part 5}: (c) and (h) are equivalent. We
 already have all the equivalences, except that of (c): The implication (h)$\implies$(c)
 is trivial by definition of a cylinder set. But we also get (c)$\implies$(d) from
 the following reasoning: A cylinder set is a finite Boolean combination of weakly
 stable sets. These are retro-compact as (h)$\implies$(d) due to part 3. But
 by~\cite[Cor.~2.3.4]{EGA1} finite Boolean combinations of open, retro-compact
 sets stay retro-compact.
\end{proof}

\section{Proof of Theorem~\ref{thm:specquasicompact}}
\label{sec:proofs2}

 In the following let $(I,\leq)$ be a net and $\big( X_i, f_j^i \big)$ be an inverse
 system of topological spaces over $I$. The inverse limit of this system is $X_{\infty}$.
 We denote the canonical projections by $\greekp_k \colon X_{\infty} \to X_k : \big(
 x_i \big) \mapsto x_k$. Note that $F^L = \invlim F^I$, where $I$ runs over the finite
 subsets of $L$, canonically. In this sense the notions \emph{weakly stable} and
 \emph{cylinder set}, that have been used in \ref{thm:affineresult}, are just a
 special case of the following general definitions.

\begin{definition} \label{def:weaklystable}
 A subset $U \setin X_{\infty}$ is said to be {\bf weakly stable}, if there is some
 $k \in I$ and an open $A \setin X_k$ such that $U = \greekp_k^{-1}(A)$. The collection
 of weakly stable subsets of $X_{\infty}$ will be denoted by $\scriptW$.
\end{definition}

\begin{definition} \label{def:constructible}
 Let $X$ be a topological space, a subset $W \setin X$ is said to be
 {\bf quasi-constructible}, if it is a finite Boolean combination of open subsets
 of $X$. And $W \setin X$ is said to be {\bf constructible}, if it is a finite
 Boolean combination of subsets of $X$, that are both, open and retro-compact.
\end{definition}

 These notions are taken from \cite[Def.~2.3.2]{EGA1}. Note that in a noetherian
 topological space $X$ any quasi-constructible set already is constructible. 

\begin{definition} \label{def:cylinderset}
 A subset $C \setin X_{\infty}$ is said to be a {\bf cylinder set}, if it satisfies
 one of the following equivalent conditions:
 \begin{enumerate}
  \item[(a)]
	 $C$ is a finite, Boolean combination of open, weakly stable subsets (i.e.~$C$ is
	 an arbitrary combination of unions, intersections and complements of finitely many
	 sets, taken from $\scriptW$).
	\item[(b)]
	 There are some $i(m) \in I$ and some locally closed subsets $L_m \setin X_{i(m)}$
	 (where $m \in 1,\ldots,n$) such that $C$ is a finite union of the form
	 \[ C \ = \ \bigcup_{m=1}^n \greekp_{i(m)}^{-1} \big( L_m \big). \]
	\item[(c)]
	 There is a quasi-constructible subset $W \setin X_i$ on some level $i \in I$,
	 such that $C$ is the preimage $C = \greekp_i^{-1}(W)$ of this set.
 \end{enumerate}
\end{definition}

 As $X_{\infty}$ is equipped with the initial topology, $\scriptW$ is a basis of this
 topology. In both our cases open and weakly stable is equivalent to open and
 retro-compact, hence (a) is equivalent to saying $C$ is constructible in $X_{\infty}$.

\begin{definition} \label{def:quasistable}
 An open subset $U \setin X_{\infty}$ is said to be {\bf quasi-stable} if every
 cover by open base sets is derived from a cover of bounded level of stability: Let
 $W_{\greekl} \in \scriptW$ (where $\greekl \in \greekL$) be a family of open base
 sets, that is for any $\greekl \in \greekL$ there are some $i(\greekl) \in I$ and
 $A_{\greekl} \setin X_{i(\greekl)}$ open, such that $W_{\greekl} = \greekp_{i(\greekl
 )}^{-1}(A_{\greekl})$. And for some $i \in I$ let us denote the set of all $\greekl$
 such that $W_{\greekl}$ belongs to a level $i$ of stability by
 \[ \greekL[i] \ := \ \set{\greekl \in \greekL \mid \E B_{\greekl} \setin X_i
    \mbox{ open} \ \colon \ W_{\greekl} = \greekp_{i}^{-1} \big( B_{\greekl} \big)}. \]
 Then, if the $W_{\greekl}$ cover $U$, there already is some $k \in I$ such that
 $U$, is covered by the $W_{\greekl}$ with $\greekl \in \greekL[k]$ only. Formally
 we have the implication
 \[ U \ = \ \bigcup_{\greekl \in \greekL} W_{\greekl} \ \ \implies \ \ 
    \E \, k \in I \ \colon \ U \ = \ \bigcup_{\greekl \in \greekL[k]} W_{\greekl}. \]
\end{definition}

\begin{lemma} \label{lem:spechomeomorph}
 Let $(I,\leq)$ be a net and $\big( R_i, \greekph_i^j \big)$ be a direct system of
 commutative rings over $I$ and $R_{\infty}$ the associated direct limit. For $i
 \leq j \in I$ let $X_i := \spec(R_i)$ and $f_j^i := \spec(\greekph_i^j)$, the
 inverse limit of $\big( X_i, f_j^i \big)$ is denoted by $X_{\infty}$. Then we obtain
 a homeomorphism, by virtue of
 \[ s \ \colon \ \spec \big( R_{\infty} \big) \, \isoto \, X_{\infty} \ : \
    \scriptq \, \mapsto \, \big( \scriptq \cap R_i \big). \]
 If $\greeke_j \colon R_j \to R_{\infty} : a \mapsto [j,a]$ is the structural
 homomorphism of the direct limit $R_{\infty}$, then $\spec(\greeke_i)(\scriptq)
 = \scriptq \cap R_i$ and thereby $s$ makes the following diagram commute:
 \[ \begin{matrix}
		 \scriptq          & \spec \big( R_{\infty} \big)     & \mapatop{s}
		                   & X_{\infty}                       & \big( \scriptp_i \big)     \\
		 \mapsdown         & \mapdown{\rm{Spec}(\greeke_j)}   & 
			 	  			       & \mapdown{\greekp_j}              & \mapsdown                  \\
		 \scriptq \cap R_j & \spec \big( R_j \big)            & =
		  	               & X_j                              & \scriptp_j
 \end{matrix} \]
\end{lemma}

 This result is taken from~\cite[Corollaire 8.2.10]{EGA4.3}. The main task here
 is to prove the surjectivity of $s$: Given $\big( \scriptp_i \big) \in X_{\infty}$
 the set $\scriptq := \{[j,p] \in R_{\infty} \mid j \in I, p \in \scriptp_j\}$ is a
 prime ideal of $R_{\infty}$, that satisfies $\scriptq \cap R_i = \scriptp_i$ for
 any $i \in I$. The homeomorphism does not follow from the classical equivalence
 of the categories of affine schemes and commutative
 rings~\cite[\href{https://stacks.math.columbia.edu/tag/01YW}{Tag 01YW}]{SP}.
 A priori it is not clear, that the Zariski topology of $\spec \big( R_{\infty}
 \big)$ is the limit (i.e.~initial) topology of $X_{\infty}$.
 
\begin{corollary} \label{cor:speccorrespondence}
 We continue with the situation of \ref{lem:spechomeomorph}. For any $j \in I$ and
 any $a \in R_j$ let $B(a)$ denote the principal open set of $a$ in $X_j$. Respectively
 let $\scripta$ be an ideal of $R_j$, then we get the identities
 \begin{eqnarray*}
  \greekp_j^{-1} \Big( B(a) \Big) 
	& = & s \left( B \big( [j,a] \big) \right),                                          \\
	\greekp_j^{-1} \Big( \setV(\scripta) \Big)
	& = & s \left( \setV \Big( \scripta R_{\infty} \Big) \right).
 \end{eqnarray*}
\end{corollary}

\begin{proof}[Proof of~\ref{cor:speccorrespondence}]
 Let $(\scriptp_i) := s(\scriptq)$ for some prime ideal $\scriptq \ideal R_{\infty}$.
 For the first identity we have to show, that $\scriptq \in B([j,a])$ is equivalent
 to $(\scriptp_i) \in \greekp_i^{-1} \big( B(a) \big)$. But the latter is $\scriptp_j
 \in B(a)$, in other words $a \nonin \scriptp_j$. As $\scriptp_j = \scriptq \cap R_j
 = \greeke_j^{-1}(\scriptq)$ we find that $a \nonin \scriptp_j$ is equivalent to
 $[j,a] = \greeke_j(a) \nonin \scriptq$. And this again is $\scriptq \in B \big(
 [j,a] \big)$.  
	
 For the second we have to prove the identity of the pre-image $\greekp_j^{-1} \big(
 \setV( \scripta) \big)$ and the image of $\setV \big( \scripta R_{\infty} \big)$
 under $s$. The sets involved are given to be
 \begin{eqnarray*}
  \greekp_j^{-1} \Big( \setV(\scripta) \Big)
	& = & \set{\big( \scriptp_i \big) \mid \scripta \setin \scriptp_j},                \\
  s \left( \setV \Big( \scripta R_{\infty} \Big) \right)
  & = & \set{\big( \scriptq \cap R_i \big) \mid \scripta R_{\infty} \setin \scriptq}.
 \end{eqnarray*}
 If we start with $\scripta R_{\infty} \setin \scriptq$, then we also have
 $\big( \scripta R_{\infty} \big) \cap R_j \setin \scriptq \cap R_j = \scriptp_j$.
 Thus $s \big( \setV(\scripta R_{\infty}) \big)$ is contained in $\greekp_j^{-1}
 \big( \setV(\scripta) \big)$. Conversely, if we start with $\scripta \setin \scriptp_j$
 we first choose $\scriptq := s^{-1} \big( (\scriptp_i) \big)$ which is $\scriptq
 \cap R_i = \scriptp_i$ for any $i \in I$. Then we find that $\scripta R_{\infty}
 \setin \scriptp_j R_{\infty} = \big( \scriptq \cap R_j \big) R_{\infty} \setin
 \scriptq$.
\end{proof}

\begin{corollary} \label{cor:quasicompact}
 We continue with the situation of \ref{lem:spechomeomorph} and consider an arbitrary
 subset $U \setin X_{\infty}$. Then the following two statements are equivalent:
 \begin{enumerate}
  \item[(a)]
	 $U$ is open and quasi-compact.
	\item[(b)]
	 $U = \greekp_i^{-1}(A)$ for some $i \in I$ and $A \setin X_i$ open and
	 quasi-compact.
 \end{enumerate}
\end{corollary}

\begin{proof}[Proof of~\ref{cor:quasicompact}]
 The implication (a)$\implies$(b) is~\cite[Cor.~8.2.11]{EGA4.3}. We verify
 (b)$\implies$(a) only: For any $i \in I$ the principal open sets $\scriptB_i
 := \set{B(a) \mid a \in R_i}$ form a basis of the topology of $X_i$ so as $A
 \setin X_i$ is open there are $a_{\greekl} \in R_i$ (where $\greekl \in \greekL$)
 such that $A$ is covered by the $B(a_{\greekl})$. But as $A$ also is quasi-compact
 there is a finite subset $\greekO \setin \greekL$ such that the $\greekl \in \greekO$
 suffice to cover $A$:
 \[ A \ = \ \bigcup_{\greekl \in \greekL} B \big( a_{\greekl} \big) \ = \
    \bigcup_{\greekl \in \greekO} B \big( a_{\greekl} \big). \]
 By \ref{cor:speccorrespondence}.(i) we have $\greekp_i^{-1} \big( B(a_{\greekl})
 \big) = s \big( B([i,a_{\greekl}]) \big)$. As $B([i,a_{\greekl}])$ is a principal
 open set of $R_{\infty}$ it is quasi-compact (this is true for all commutative rings).
 And as $s$ is a homeomorphism this means $\greekp_i^{-1} \big( B(a_{\greekl}) \big)$
 is quasi-compact, too. But
 \[ U \ = \ \greekp_i^{-1}(A) \ = \ \greekp_i^{-1} \left( \bigcup_{\greekl \in \greekO}
    B \big( a_{\greekl} \big) \right) \ = \ \bigcup_{\greekl \in \greekO} \greekp_i^{-1}
		\left( B \big( a_{\greekl} \big) \right), \]
 such that $U$ is a finite union of these sets. But a finite union of quasi-compact
 sets stays quasi-compact. And $U$ clearly is open, as $\greekp_i$ is continuous and
 $A$ is open.
\end{proof}

\begin{proof}[Proof of Theorem~\ref{thm:specquasicompact}]
 {\bf Part 1}: (a) to (c) are equivalent. Hereby (a)$\implies$(b) and (b)$\implies$(c)
 are generally true, see part 2 of the proof of \ref{thm:affineresult} for a few
 comments on this. In (c)$\implies$(a) we have $U = \greekp_k^{-1}(A)$ for some open
 set $A \setin X_k$. By assumption $R_k$ is a noetherian ring and thereby $X_k$ a
 noetherian topological space. Hence $A \setin X_k$ is quasi-compact and by
 \ref{cor:quasicompact} this makes $U = \greekp_k^{-1}(A)$ quasi-compact.
	
 {\bf Part 2}: (c) and (d) are equivalent. In (c)$\implies$(d) we consider an open
 $V \setin X_{\infty}$ that also is quasi-compact. By part 1, $V$ is weakly-stable
 and by assumption $U$ is weakly stable, too. Hence $U \cap V$ is weakly stable and
 thereby quasi-compact. We prove (d)$\implies$(a) next: $X_{\infty}$ is homeomorphic
 $\spec \big( R_{\infty} \big)$ by \ref{lem:spechomeomorph} and thereby $X_{\infty}$
 is a quasi-compact topological space. In this case retro-compact implies quasi-compact.

 {\bf Part 3}: (c) and (e) are equivalent. The implication (c)$\implies$(e) is trivial
 by definition of cylinder sets, so we turn to (e)$\implies$(c): Let $U = \greekp_i^{
 -1}(W)$ for some quasi-constructible subset $W \setin X_i$. That is $W = L_1 \cup
 \dots \cup L_n$ for some locally closed $L_m$. Say $L_m = U_m \cap C_m$ for some
 $U_m \setin X_i$ open and $C_m \setin X_i$ closed. As $X_i$ is a noetherian space,
 any $U_m$ is quasi-compact. Hence $\greekp_i^{-1}(U_m) \setin X_{\infty}$ is open and
 quasi-compact, by \ref{cor:quasicompact}. And $\greekp_i^{-1}(C_m) \setin X_{\infty}$
 is closed, such that $\greekp_i^{-1}(L_m) = \greekp_i^{-1}(U_m) \cap \greekp_i^{-1
 }(C_m) \setin \greekp_i^{-1}(U_m)$ is closed. As $\greekp_i^{-1}(U_m)$ is quasi-compact,
 the closed subset $\greekp_i^{-1}(L_m)$ inherits this. But then $U = \greekp_i^{-1}(W)
 = \greekp_i^{-1}(L_1) \cup \ldots \cup \greekp_i^{-1}(L_n)$ is quasi-compact, as it
 is a finite union of such sets. By part 1 yields (a) is (c).
 
 {\bf Part 4}: (a) and (f) are equivalent. Generally a subset $U$ of the prime spectrum
 of $R_{\infty}$ is open and quasi-compact, iff $U = X_{\infty} \setminus \setV(\scriptu)$
 for some finitely generated ideal $\scriptu \ideal R_{\infty}$. Starting in (a) we
 find this $\scriptu$ and by \ref{pro:limcompact}.(iii) finitely generated implies
 quasi-finite. Starting in (f) we have $U = X_{\infty} \setminus \setV(\scriptu)$ for
 some $\scriptu = \scripta R_{\infty}$, where $i \in I$ and $\scripta \ideal R_i$. As
 $R_i$ was assumed to be a noetherian ring $\scripta$ is finitely generated and hence
 $\scriptu$ inherits being finitely generated. By the general argument this implies
 $U$ to be quasi-compact.
\end{proof}

 Comparing this to EGA the implication (a)$\implies$(d) is \cite[Prop.~1.2.7]{EGA4.1}
 and the implication (a)$\implies$(c) is \cite[Corollaire 8.2.11]{EGA4.3}. De Fernex
 and Docampo noticed in the introduction to Section 10 of \cite{dFDo}, that
 \cite[Corollaire 8.3.11]{EGA4.3} implies, that cylinder sets are precisely the
 finite Boolean combinations of open, retro-compact sets.

\begin{lemma} \label{lem:ultraconstruct} \cite[Corollaire 1.9.9]{EGA4.1}
 Let $R$ be a commutative ring and $X = \spec(R)$ be its prime spectrum. Let $W$ be
 the union of constructible sets~$W_{\greekl}$. If $W$ is con-\break structible,
 then $W$ is the union of finitely many $W_{\greekl}$.
\end{lemma}

 Note that this need not be true, if we replace \emph{constructible} by
 \emph{quasi-constructible}. Example~\cite[\href{https://stacks.math.columbia.edu/tag/01K8}{Tag 01K8}]{SP}
 of the stacks project depicts a situation in which the modified statement
 fails.

\begin{proof}[Proof of Lemma~\ref{lem:ultraconstruct}]
 The following proof is inspired by \cite[Lemma 3.1]{FFL}. In the claim $X$ carries
 the usual Zariski-topology. We will equip $X$ with the \emph{patch topology} now:
 The smallest topology, that contains the open base sets $B(a)$ and the closed sets
 $\setV(b)$ for $a$, $b \in R$. Thereby any $B(a)$ is patch-clopen. And as open,
 quasi-compact sets are finite unions of such sets, they are patch-clopen, too.
 From this we find, that any constructible subset $W \setin X$ is patch-clopen.
 A base of this topology is given by
 \[ \scriptB \ := \ \set{B(a) \cap \setV(b_1) \cap \dots \cap \setV(b_n)
    \mid a, b_1, \ldots, b_n \in R} \]
 Let $\scriptU$ be an ultrafilter in $X$. Then $\scriptp_{\smallscriptU} := \set{b
 \in R \mid \setV(b) \in \scriptU}$ is a prime ideal of $R$: The filter properties
 turn it into a proper ideal and the ultrafilter property yields being prime. From
 the filter property of $\scriptU$ we also get the following implication
 \[ \scriptp_{\smallscriptU} \in B \in \scriptB \ \implies \ B \in \scriptU \]
 Consider any patch-neighborhood $N$ of $\scriptp_{\smallscriptU}$, as $\scriptB$
 is a base of the patch topology there is some $B \in \scriptB$ such that $\scriptp_{
 \smallscriptU} \in B \setin N$. By the implication above we have $B \in \scriptU$.
 Then $B \setin N$ yields $N \in \scriptU$, as a filter is upward closed. Thus the
 entire neighborhood filter $\scriptN(\scriptp_{\smallscriptU})$ is contained in
 $\scriptU$. This means $\scriptU \to \scriptp_{\smallscriptU}$ and by the ultrafilter
 criterion \cite[Thm.~2.1]{Moor} this makes $X$ quasi-compact in the patch-topology.
 Now consider $W \setin X$ which is the union of the constructible sets $W_{\greekl}$
 in the Zariski topology. As these are patch-open, we have an open cover of $W$
 \[ W \ = \ \bigcup_{\greekl \in \greekL} W_{\greekl} \]
 By assumption $W$ is quasi-constructible, too, hence it is patch-closed. As $X$ is
 quasi-compact in the patch-topology, so is the closed subset $W \setin X$. Thus
 the open cover has a finite subcover.
\end{proof}
	
\begin{corollary} \label{cor:constructiblecompact}
 In the situation of Theorem~\ref{thm:specquasicompact}, let $C$ be the union of cylinder
 sets~$C_{\greekl}$ in $X_\infty$. If $C$ is a cylinder set, then $C$ is the union of
 finitely many $C_{\greekl}$.
\end{corollary}

\begin{proof}[Proof of Corollary~\ref{cor:constructiblecompact}]
 By property (a) in Definition \ref{def:cylinderset} any cylinder set is constructible
 in $X_{\infty}$. And by Lemma~\ref{lem:spechomeomorph} $X_{\infty}$ carries the
 topology of a prime spectrum. Thus the claim immediately follows from Lemma
 \ref{lem:ultraconstruct}.
\end{proof}

\section{Examples and Discussion}
\label{sec:counterexamples}

\begin{example} \label{xmp:notquasicompact}
 For the field $F = \overline{\setQ}$ of algebraic numbers, the topological space
 $F^{\setN}$ is not quasi-compact. More generally, consider any countable, algebraically
 closed field $F = \set{a_1,a_2,\ldots}$ and the $F$-algebra epimorphism onto the
 function field
 \[ \greekee \colon \, F[t_i \mid i \in \setN] \onto F(s) \ : \ t_i \, \mapsto \,
    \begin{cases} s, & \quad i = 0, \\ \frac{1}{s-a_i}, & \quad i \geq 1. \end{cases} \]
 Note that surjectivity is guaranteed by $F$ being algebraically closed. As~$F(s)$
 is a field, $\scriptm := \kn(\greekee)$ is a maximal ideal in $F[t_i \mid i \in
 \setN]$. Consider its restrictions
 \[ \scriptp_i \ := \ \scriptm \cap F[t_0,t_1,\ldots,t_i] \ \subseteq \ F[t_0,\ldots, t_i]. \]
 As $\scriptm$ is a proper ideal, each $\scriptp_i$ is so, too. Hence $\setV(\scriptp_i)
 \setin F^I$ is non-empty, by the weak Nullstellensatz. For $I = \{0,\ldots,i\}$ consider
 the projection $\greekp_i \colon F^{\setN} \to F^I$. Then the following closed set is
 non-empty, as well:
 \[ C_i \ := \ \greekp_i^{-1} \Big( \setV(\scriptp_i) \Big) \ \setin \ F^{\setN}. \]
 By construction, $\scriptp_i = \scriptm \cap F[t_0,\ldots,t_i] \setin \scriptm \cap
 F[t_0, \ldots,t_i,t_{i+1}] = \scriptp_{i+1}$, such that $\setV(\scriptp_{i+1}) \setin
 \setV(\scriptp_i)$ and hence $C_{i+i} \setin C_i$. That is, the $C_i$ form a descending
 chain of closed sets. We claim
 \[ C \ := \ \bigcap_{i \geq 1} C_i \ = \ \emptyset. \]
 Suppose there was an $x$ in $C$. Then $x = (x_0,x_1,\ldots) \in F^{\setN}$ and, as
 $x_0 \in F$, we have $x_0 = a_k$ for some $k \geq 1$. Consider the polynomial
 \[ f_k \ := \ (t_0 - a_k) t_k - 1  \ \in \ \scriptm \cap F[t_0,\ldots,t_k] \ = \ \scriptp_k. \]
 Then $f_k(x_0,\ldots,x_k) = -1 \noneq 0$, hence $x \nonin C_k$; a contradiction.
 Consequently $C$ is empty. Let $U_i := F^{\setN} \setminus C_i$ be the complement
 of $C_i$. By the above, we obtain
 \[ \bigcup_{i \geq 1} U_i \ = \ \bigcup_{i \geq 1} F^{\setN} \setminus C_i \ = \
    F^{\setN} \setminus \bigcap_{i \geq 1} C_i \ = \ F^{\setN}. \]
 Thus the $U_i$ form an open cover of $F^{\setN}$. Moreover, as the $C_i$ form a
 descending chain, the $U_i$ form an ascending chain. If $F^{\setN}$ were quasi-compact,
 then we would have $F^{\setN} = U_m$ for some $m \in \setN$. Since $C_m \ne \emptyset$,
 this would lead to a contradiction:
 \[ F^{\setN} \ = \ U_m \ = \ F^{\setN} \setminus C_m \ \ne \ F^{\setN}. \]
\end{example}

\begin{remark}
 One can turn $F^{\setN}$ into another topological space by regarding it not as the
 \emph{inverse} limit $\invlim F^n$, but as the \emph{direct} limit $\dirlim F^n$,
 where $F^n \setin F^{n+1}$ is defined by $x_{n+1} = 0$. Then $F^{\setN}$ is a basic
 example of an ind-scheme and will never be quasi-compact, as the inclusions $F^n
 \setinn F^{n+1}$ are strict; see~\cite[Lem.~1.22]{Rich}.
\end{remark}

 Drawing on Example~\ref{xmp:notquasicompact}, we point out limits on generalizing
 Chevalley’s classical theorem: morphisms of finite type between noetherian schemes
 map constructible sets to constructible sets.

\begin{example} \label{xmp:notstable}
 For any countable algebraically closed field $F = \{a_1,a_2,\ldots\}$, consider the
 inclusion of polynomial rings $F[t_1] \subseteq F[t_i \mid i \in \setN]$ and let
 $\greekp \colon X_\infty \to X_1$ be the corresponding morphism of the associated
 prime spectra. As in Example~\ref{xmp:notquasicompact}, look at the epimorphism
 \[ \greekee \colon \, F[t_i \mid i \in \setN] \onto F(s) \ : \ t_i \, \mapsto \,
    \begin{cases} s, & \quad i = 0, \\ \frac{1}{s-a_i}, & \quad i \geq 1 \end{cases} \]
 and the maximal ideal $\scriptm := \kn(\greekee)$ in $F[t_i \mid i \in \setN]$. 
 We claim $\greekp(\scriptm) = 0$. Indeed, $f \in \scriptm \cap F[t_1]$ means $f
 \in F[t_1]$ and $\greekee(f) = 0$. The latter is $f \big( 1/(s-a_1) \big) = 0$.
 But as $\{1/(s-a_1)\} \setin F(s)$ is algebraically independent over $F$, this
 implies $f = 0$.
 
 As $\scriptm$ is maximal, $\set{\scriptm} \setin X_\infty$ is closed and finite
 (hence retro-compact). In $X_1 = \spec \big( F[t_1] \big)$ the quasi-constructible
 subsets are precisely the finite subsets of closed points and their complements,
 hence precisely the closed or open subsets. However, $\greekp(\set{\scriptm})
 = \set{0} \setin X_1$ is neither closed (it contains the generic point) nor open
 (it is not cofinite) and hence not quasi-constructible.
\end{example}

 Recall the situation of Theorem~\ref{thm:specquasicompact}, $X_{\infty}$ is the
 inverse limit of prime spectra of noetherian rings. If $U \setin X_{\infty}$ is an
 open, quasi-compact set, then $U = \greekp_k^{-1}(A)$ for some $k \in I$ and
 $A \setin X_k$ open. In particular $U = \greekp_k^{-1} \big( \greekp_k(U) \big)$,
 that is $U$ satisfies the stability condition (f) in affine spaces, as given in
 \ref{thm:affineresult}. The following example points out that quasi-compactness
 and this stability condition cannot be equivalent, for prime spectra of general
 commutative rings.

\begin{example} \label{xmp:notequivalent}
 Let $S := \setC[t_1,t_2,\dots]$ and $R_i := \setC[t_1,\dots,t_i][t_j^2, t_j^3 \mid
 j > i]$ for any $i \in \setN$. Then it is clear that $R_0 \setin R_1 \setin
 \ldots$ forms an ascending chain of rings, $S$ is the union of all the $R_i$ and
 hence $S = R_{\infty}$ is the direct limit, canonically. As $R_i = A_i[t_i]$ and
 $R_{i-1} = A_i[t_i^2, t_i^3]$ for a suitable ring $A_i$ we see, that $R_{i-1} \setin
 R_i$ is a subintegral extension, from \cite[Example~2.6]{Vitu}. In particular the
 corresponding maps $X_i \to X_{i-1}$ are bijective. Consequentially the canonical
 projections $X_{\infty} \to X_i$ are bijective, as well. In particular we have
 (for any $U \setin X_{\infty}$ and any $i \in \setN$)
 \[ U \;\; = \;\; \greekp_i^{-1} \big( \greekp_i(U) \big). \]
 Yet $S$ contains the maximal ideal $\scriptm := \gen{t_1,t_2,\dots}$. We let $C
 := \set{\scriptm}$ and $U = X_{\infty} \setminus C$. Then $U$ is open, but \emph{not}
 quasi-compact, as $\scriptm$ contains infinitely many variables $t_j$: For any
 $1 \leq n \in \setN$ the ideal $\gen{t_1,\dots,t_n}$ lies in $U$ but not in the
 union $B(t_1) \cup \ldots \cup B(t_n)$. By the letter $U$ satisfies (f) in
 Theorem \ref{thm:affineresult}, but neither (a) in Theorem \ref{thm:affineresult}
 nor (a) in Theorem \ref{thm:specquasicompact}.
\end{example}

\begin{remark} \label{rem:closingremarks}
 Let $F$ be an algebraically closed field, $L$ a set and $\scriptI$ be the set of
 all finite subsets of $L$. We summarize key facts showing that inverse limits
 of maximal spectra and of prime spectra can exhibit different quasi-compactness
 properties:
 \begin{enumerate}
  \item
	 The affine space $F^L$ is the inverse limit of $F^I$ where $I \in \scriptI$.
	 Hereby $F^I$ identifies with the maximal spectrum of $F[I]$.
  \item
   $F^L$ is quasi-compact when $\abs{L} < \abs{F}$ according to Theorem~\ref{thm:affineresult},
	 but need not be so, if $\abs{L} = \abs{F}$, see Example~\ref{xmp:notquasicompact}.
  \item
   Likewise the prime spectrum $\spec \big( F[L] \big)$ is the inverse limit of the
	 prime spectra $\spec \big( F[I] \big)$ where $I \in	\scriptI$, due to
	 Lemma~\ref{lem:spechomeomorph}.
  \item
	 $\spec \big( F[L] \big)$ always is quasi-compact, this is true for any commutative
	 ring.
 \end{enumerate} 
\end{remark}

\section*{Acknowledgments}

The author would like to express his sincere gratitude to Professor Jürgen Hausen
for his support and invaluable advice.

\end{document}